\documentclass{article}
\usepackage[a4paper, lmargin=3cm, rmargin=3cm, marginpar=3.0cm]{geometry}
\usepackage[utf8]{inputenc}
\usepackage{authblk}

\usepackage{amsmath}
\usepackage{mathrsfs}
\usepackage{mathtools}
\usepackage{amssymb}
\usepackage{bbm}
\usepackage{csquotes}

\usepackage{parskip}
\usepackage{physics}
\usepackage{amsthm}
\usepackage[normalem]{ulem}

\usepackage{hyperref}

\DeclareUnicodeCharacter{2113}{$\ell$}
\usepackage[english]{babel}
\usepackage[backend=biber, url=false]{biblatex}

\usepackage{xcolor}
\usepackage{todonotes}

\newtheorem{definition}{\textit{Definition}}
\newtheorem{lemma}{\textbf{Lemma}}
\newtheorem{remark}{\textbf{Remark}}

\newtheorem{theorem}{\textbf{Theorem}}

\allowdisplaybreaks

\DeclarePairedDelimiterX{\set}[1]{\lbrace}{\rbrace}{\def\given{\;\delimsize\vert\;\allowbreak}#1}
\DeclarePairedDelimiterX{\sip}[2]{[}{]}{#1\,\delimsize\vert\,\mathopen{}#2}

\newcommand{\B}{\mathcal{B}}
\newcommand{\di}{\diamond}

\newcommand{\e}{\text{e}}

\newcommand{\F}{\mathcal{F}}
\renewcommand{\H}{\mathcal{H}}
\newcommand{\J}{\mathcal{J}}

\newcommand{\N}{\mathbb{N}}

\newcommand{\R}{\mathbb{R}}
\renewcommand{\S}{\mathcal{S}}

\newcommand\extrafootertext[1]{%
    \bgroup
    \renewcommand\thefootnote{\fnsymbol{footnote}}%
    \renewcommand\thempfootnote{\fnsymbol{mpfootnote}}%
    \footnotetext[0]{#1}%
    \egroup
}

\title{\vspace{-2.5cm}Characterization of the reproducing structure of the Bessel potential spaces beyond $p=2$.}

\author[1]{Tjeerd Jan Heeringa\thanks{t.j.heeringa@utwente.nl}}
\affil[1]{\small{Mathematics of Imaging \& AI, University of Twente, The Netherlands }}
\date{}

\begin{document}

\maketitle

\begin{abstract}
Reproducing kernel Hilbert spaces are uniquely characterized by their kernel, but reproducing kernel Banach spaces (RKBS) are not. However, a characterization of which RKBS admit a given kernel as reproducing kernel is lacking. This work provides such a characterization for the well-known Bessel potential / Matèrn kernel, a widely used covariance kernel for Gaussian processes which is the reproducing kernel of the Bessel potential space $H^{s,2}(\R^d)$ when $s>d/2$. Concretely, this work characterizes the pairs of Bessel potential spaces $H^{u,p}(\R^d),H^{v,q}(\R^d)$ which have this kernel. 
\end{abstract}
\extrafootertext{\textbf{keywords: }Reproducing Kernel Banach Space, RKBS, Bessel potential space, Matèrn kernel, reproducing kernel, Reproducing Kernel Hilbert Space, RKHS}

\section{Introduction}
Reproducing kernel Hilbert spaces (RKHS) are spaces characterized by a reproducing kernel. In these spaces, any point evaluation can be written as an inner-product with this kernel, i.e. if we denote the kernel with $K$, then
\begin{equation}\label{eq:reproducing_kernel_property_rkhs}
    f(x) = \braket{K(x,\cdot)}{f}_\H
\end{equation}
for all $f$ in the RKHS and for all $x$ from the set $X$ over which the RKHS is defined. These kernels $K$ are positive definite and symmetric functions. They can be interpreted as measuring the similarity between two points $x$ and $y$, as the kernel can be expressed as an inner product in the RKHS:
\begin{equation}
    K(x,y) = K(\cdot,y)(x) = \braket{K(x,\cdot)}{K(\cdot,y)}_\H
\end{equation}
This enables results to be explained geometrically. Another benefit of using RKHSs is that optimization problems of the form
\begin{equation}
    \min_{f\in \H}\sum_{n=1}^N\abs{f(x_n)-y_n}^2 +\lambda \norm{f}^2_{\H}
\end{equation}
have a solution of the form
\begin{equation}
    f(x) = \sum_{n=1}^M c_n K(x,x_n)
\end{equation}
with $M\leq N$. Substituting this sparse solution back into the optimization problem turns it from an infinite-dimensional problem into a finite-dimensional one. These properties have lead to various schemes for regression and classification on finite data.

There are at least two benefits to working with non-Hilbert Banach spaces. First, Hilbert spaces have the property that any two of them with the same dimension are isometrically isomorphic. This means that there are limited properties of the spaces that we can leverage to improve our optimization problems. Second, some problems have an intrinsic structure that is unable to be embedded in a Hilbert space.

In 2009, Zhang et al. \cite{zhang_reproducing_2009} introduced reproducing kernel Banach spaces (RKBS) as a reflexive Banach space of functions for which the dual is isometric to a Banach space of functions over the same set with the evaluation functionals being bounded on Banach spaces of functions. In 2013, Song et al. \cite{song_reproducing_2013} defined an RKBS for a space with an $\ell^1$-type norm via admissible kernels. In 2014, Georgiev et al \cite{demyanov_construction_2014} relaxed the requirements on both Banach spaces of function by making them a dual pair instead of one being isometric to the full dual of the other. Their definition is the one used in most recent works, albeit with some minor variations in how specific concepts are named. For more information, we recommend reading the review work by Lin et al. \cite{lin_reproducing_2022}, which shows, among others, how the different RKBS considered in the literature fit into this formulation. 

We mentioned before that RKHS are characterized by their kernel, i.e. every RKHS has a unique kernel and every kernel defines an RKHS. For RKBS, this characterization fails in both directions: A given RKBS pair $(\B,\B^\di)$ may admit multiple kernels, and a given kernel can be the kernel for multiple RKBS pairs $(\B,\B^\di)$. This begs the following question:
\begin{displayquote}
    Given a fixed kernel K, what are possible distinct dual pairs of RKBSs $\B,\B^\di$ which have $K$ as kernel? 
\end{displayquote}
To get an understanding of what dual pairs this could be, we restrict ourselves to a particular one: the Bessel potential kernel $K_s(x,y)=G_{2s}(x-y)$. This kernel is also known as the Matèrn kernel and is widely used a covariance kernel for Gaussian processes. It is the reproducing kernel of the Bessel potential space $H^{s,2}(\R^d)$ for $s>d/2$. Bessel potential spaces $H^{s,p}(\R^d)$ are the Fourier-based generalization of the Sobolev spaces $W^{m,p}(\R^d)$. These are well-studied spaces and as such they provide a good starting point.

\subsection{Related work}
Due to Plancherel, the norm of the Bessel potential space satisfies
\begin{equation}
    \norm{f}_{H^{s,2}(\R^d)}^2 = \int_{\R^d} (1+\abs{\xi}^2)^s\abs{\hat{f}(\xi)}^2 d\xi = \int_{\R^d} \frac{\abs{\hat{f}(\xi)}^2}{\widehat{K_s}(\xi)}d\xi
\end{equation}
Fasshauer et al. \autocite{fasshauer_solving_2015} extend this Fourier-based formulation to the space
\begin{equation}
    \B^p_{s} = \set[\Bigg]{ f \in C(\R^d)\cap \S'(\R^d) \given \norm{f}_{\B^p_{s}}:= \bigg(\int_{\R^d}\frac{|\hat{f}(\xi)|^{p'}}{\widehat{K_{s}}(\xi)}d\xi\bigg)^{1/p'} < \infty}
\end{equation}
for $1<p<\infty$, $s>d/2$ and $p'=p/(p-1)$, and show that $\B^p_{s},\B^{p'}_{s}$ is an RKBS pair with kernel $K_s$. These pairs $\B^p_{s},\B^{p'}_{s}$ are not comparable to pairs $H^{u,p}(\R^d),H^{v,q}(\R^d)$, because we only have
\begin{subequations}
\begin{align}
    \B^p_{s} &\hookrightarrow H^{2s/p',p'}(\R^d) & 1 \leq p \leq 2 \\
    H^{2s/p',p'}(\R^d) &\hookrightarrow \B^p_{s} & 2\leq p < \infty
\end{align}    
\end{subequations}
by the Hausdorff-Young inequality. 

\subsection{Our contribution}
In this work, we fully characterize which $H^{s,p}(\R^d)$ spaces have 
\begin{equation}\label{eq:kernel}
    K_s(x,y)=G_{2s}(x-y)
\end{equation}
as kernel. 

\begin{theorem}\label{th:main_results}
Let $d\in \N$, $u,v,s>0$ and $1\leq p,q\leq \infty$. The spaces $H^{u,p}(\R^d),H^{v,q}(\R^d)$ are an RKBS pair with kernel $K_s$ if and only if
\begin{subequations}
\begin{align}
    \frac{1}{p}+\frac{1}{q} &\geq 1\label{eq:main_p_q_result} \\
    d/p < u &< 2s - d/p'\label{eq:main_u_result} \\
    d/q < v &< 2s - d/q'\label{eq:main_v_result} \\
    u + v &\geq 2s + d(\frac{1}{p}+\frac{1}{q}-1)\label{eq:main_u+v_result}
\end{align}    
\end{subequations}
with the inequality in \eqref{eq:main_u+v_result} being strict when $\min(p,q)=1$ and $\max(p,q)<\infty$, where $p'$ and $q'$ are the Hölder-conjugates of $p$ and $q$.
\end{theorem}

\subsection{Structure of the work}
This works starts with a short review of RKBS pairs and Bessel potential spaces. Section~\ref{sec:RKBS} provides a definition of RKBS and describes how to construct an RKBS pair. Section~\ref{sec:Bessel_potential_spaces} defines the Bessel potential spaces $H^{s,p}(\R^d)$ and provides the relevant properties that we need to construct the RKBS pairs over them. Section~\ref{sec:proof} details the proof of our contribution. Section~\ref{sec:conclusion} summarizes the work. 

\section{Reproducing kernel Banach spaces}\label{sec:RKBS}
The function spaces that we will be working with are reproducing kernel Banach spaces. These spaces have the requirement that their elements are functions.
\begin{definition}[Reproducing kernel Banach space]
A reproducing kernel Banach space (RKBS) $\B$ over a set $X$ is a Banach space of functions with bounded function evaluations, i.e. for all $x\in X$
\begin{equation}
    \abs{f(x)} \leq C_x\norm{f}_\B
\end{equation}
holds for all $f\in \B$ with the constant $C_x\geq 0$ depending on $x$ but not on $f$.
\end{definition}
Examples of RKBS are the continuous functions, reproducing kernel Hilbert spaces, Hardy spaces and neural network spaces.

The RKBS contains the "reproducing kernel" part, since function evaluation can be written as pairing of the function with a kernel. When the RKBS is a Hilbert space, this kernel is unique due to Riesz representation theorem. When the RKBS is not a Hilbert space, a second RKBS is needed and the kernel is unique for this RKBS pair.

\begin{definition}\label{def:dual_pair}
A dual pair of Banach spaces $\B,\B^\diamond$ is two Banach spaces $\B$ and $\B^\diamond$ with a continuous bilinear map 
\begin{equation}
    \braket{\cdot}{\cdot}\colon\B^\diamond\times \B\to \R, \;(g,f)\mapsto \braket{g}{f}
\end{equation}
that satisfies
\begin{subequations}\label{eq:separating}
\begin{align}
    \forall f\in \B\backslash\set{0} \;&\exists g\in \B^\diamond\colon \, \braket{g}{f} \neq 0\label{eq:separating_a} \\
    \forall g\in \B^\diamond\backslash\set{0} \;&\exists f\in \B\colon \;\; \braket{g}{f} \neq 0\label{eq:separating_b}
\end{align}    
\end{subequations}
This bilinear map is called the pairing corresponding to $(\B,\B^\diamond)$ or, when no confusion arises, simply the pairing.
\end{definition}

\begin{remark}
The symbol $\di$ is pronounced as diamond.    
\end{remark}

\begin{definition}
Let $X,Y$ be sets. A dual pair of Banach spaces $\B,\B^\di$ is an RKBS pair over $X,Y$ when $\B$ is an RKBS over $X$ and $\B^\di$ is an RKBS over $Y$, and there exists a function $K:X\times Y\to\R$ such that
\begin{subequations}\label{eq:reproducing_kernel}
\begin{align}
    f(x) &= \braket{K(x,\cdot)}{f} \quad \forall f\in\B,x\in X \label{eq:reproducing_kernel_B} \\
    g(y) &= \braket{g}{K(\cdot,y)} \quad \forall g\in\B^\diamond,y\in Y \label{eq:reproducing_kernel_B-di}
\end{align}    
\end{subequations}
called the kernel of $\B,\B^\di$.
\end{definition}

\begin{remark}
We use the symbol $Y$ instead of $X^\diamond$ to signal that there is no need for a bilinear pairing between $X$ and $Y$.    
\end{remark}

\section{Bessel potential spaces}\label{sec:Bessel_potential_spaces}
The Bessel potential space is subspace of the tempered distributions defined by the Bessel potential kernel\cite{grafakos_modern_2014}.

\begin{definition}
The Schwartz space or space of rapidly decaying functions is given by
\begin{equation}
    \S(\R^d) = \set{ f\in C^\infty(\R^d) \given \sup_{x\in\R^d}\abs{x^\alpha\partial^\beta f}<\infty\quad \forall \alpha, \beta \in \N^d}
\end{equation}
Its dual $\S'(\R^d)$ is the space of tempered distributions.
\end{definition}

\begin{definition}
Let $s\in\R$. The Bessel potential of order $s$ is the operator
\begin{equation}\label{eq:J}
    \J_s:\,\S'(\R^d)\to\S'(\R^d),\; u \mapsto \F^{-1}\bigg((1+\abs{\cdot}^2)^{-s/2}\F u\bigg)
\end{equation}
\end{definition}

\begin{definition}
Let $s>0$. The Bessel kernel of order $s$ is the function
\begin{equation}
    G_s:\, \R^d\to\R,\; x\mapsto \frac{1}{2^{(s-2)/2}(2\pi)^{d/2}\Gamma(\frac{s}{2})}K_{(d-s)/2}(\abs{x})\abs{x}^{(s-d)/2} 
\end{equation}
with $K_{\alpha}$ the modified Bessel function of the second kind.
\end{definition}

The potential $\J_\bullet$ is a semigroup and relates to the kernel $G_\bullet$ distributionally by
\begin{equation}\label{eq:J_G-identity}
    \J_s f = G_s\ast f
\end{equation}
where $s>0$ and $f\in \S'(\R^d)$. 

Lemma~\ref{lemma:G_integrability} characterizes the integrability of the Bessel kernel $G_s$. In particular, it shows that $G_s\in L^1(\R^d)$ for $s>d/2$. 
\begin{lemma}\label{lemma:G_integrability}
$G_s \in L^{p'}(\R^d)$ if and only if $s>d/p$.    
\end{lemma}
\begin{proof}
By Proposition 6.2.2 c) of \cite{hao_lecture_2020}, 
\begin{equation}\label{eq:G_bound_out}
    G_s(x) \lesssim \e^{-\abs{x}/2}
\end{equation}
when $\abs{x}\geq 2$ and 
\begin{equation}\label{eq:G_bound_in}
    G_s(x) \sim 1+\begin{cases}
        \abs{x}^{s-d}+O(\abs{x}^{s-d+2}) & s < d \\
        \log(2/\abs{x})+O(\abs{x}^2) & s=d \\
        O(\abs{x}^{s-d}) & s>d
    \end{cases} \sim \begin{cases}
        \abs{x}^{s-d} & s < d \\
        1+|\log(\abs{x})| & s=d \\
        1 & s>d
    \end{cases}
\end{equation}
when $\abs{x}<2$, where $O$ is a function with the property $O(\abs{t})\leq \abs{t}$ for $t\geq 0$. 

For $p=1$, $p'=\infty$. Since the upper bound in \eqref{eq:G_bound_out} is always finite, $G$ is bounded if and only if the terms squeezing it in \eqref{eq:G_bound_in} are. For $s<d$ and $s=d$, both squeeze terms explodes as they approach the origin. For $s>d$, $G_s$ is squeezed by a constant and thus clearly bounded. Hence, $G\in L^\infty$ if and only if $s>d$.

For $p>1$, $p'<\infty$ and thus the $L^{p'}$-norm is an integral. We split this integral in such a way that
\begin{equation}
    \norm{G_s}_{L^{p'}}^{p'} = \int_{\abs{x}<2} \abs{G_s}^{p'}(x)dx + \int_{\abs{x}\geq 2} \abs{G_s}^{p'}(x)dx 
\end{equation}
The far field doesn't impact the integrability, because 
\begin{equation}
    \int_{\abs{x}\geq 2} \abs{G_s}^{p'}(x)dx \lesssim \int_{\abs{x}\geq 2} \e^{-p'/2\abs{x}}dx < \infty
\end{equation}
for any $s>0$. For the near field, we need to consider the three cases separately. If $s<d$, then
\begin{equation}
    \int_{\abs{x}<2} \abs{G_s(x)}^{p'}dx \sim \int_{\abs{x}<2} \abs{x}^{(s-d)p'}dx \sim \int_0^2 r^{(s-d)p'+d-1} dr
\end{equation}
which converges if and only if
\begin{equation}
    (s-d)p'+d-1 > 0 \iff s > d/p
\end{equation}
If $s=d$, then
\begin{equation}
    \int_{\abs{x}<2} \abs{G_s(x)}^{p'}dx \sim \int_{\abs{x}<2} 1+\abs{\log{\abs{x}}}^{p'}dx \sim \int_0^2 r^{d-1}(1+\abs{\log(r)})^{p'}dr = \int_{-\infty}^{\log(2)} \e^{du}(1+\abs{u})^{p'}du < \infty
\end{equation}
If $s>d$, then
\begin{equation}
    \int_{\abs{x}<2} \abs{G_s(x)}^{p'}dx \sim \int_{\abs{x}<2} dx < \infty
\end{equation}
The $s=d$ and $s>d$ cases are consistent with the claim, because $s\geq d > d/p$. Hence, these three cases combined show the result.
\end{proof}

The Bessel potential space is constructed by applying $\J_s$ to $L^p(\R^d)$. 

\begin{definition}
Let $1 \leq p \leq \infty$ and $s\in\R$. The Bessel potential space is given by
\begin{subequations}
\begin{align}
     H^{s,p}(\R^d) &= \J_{s}L^p(\R^d) = \set{ f\in \S'(\R^d) \given \exists f_0\in L^p(\R^d): \, \J_s f_0 = f} \\
     \norm{f}_{H^{s,p}} &= \norm{\J_{-s}f}_{L^p}\label{eq:bessel_norm}
\end{align}
\end{subequations}
\end{definition}
This space is well-defined and Banach, see \cite{hytonenAnalysisBanachSpaces2016} on p. 449 (mind the sign convention of $\J_s$). It follows from Young's inequality that the functions in $H^{s,p}$ are $L^p(\R^d)$ for $s>0$. 

\begin{lemma}\label{lemma:schwartz_embbedings}
$\S(\R^d)\hookrightarrow H^{s,p}(\R^d) \hookrightarrow \S'(\R^)$. If $p=\infty$, only the second embedding is dense. Otherwise, both embeddings are dense.
\end{lemma}
\begin{proof}
For $p<\infty$, this is Proposition 5.6.4 of \cite{hytonenAnalysisBanachSpaces2016}. For $p=\infty$, we can repeat the same steps, but loose the density of $\S(\R^d)$ into $H^{s,\infty}(\R^d)$, because $\S(\R^d)$ is not dense in $L^\infty(\R^d)$.     
\end{proof}

If $p<\infty$, then the dual spaces can be identified with a Bessel potential space.

\begin{lemma}[Proposition 5.6.7 of \cite{hytonenAnalysisBanachSpaces2016}]\label{lemma:dual}
Let $1\leq p <\infty$. $(H^{s,p}(\R^d))^\ast = H^{-s,p'}(\R^d)$.
\end{lemma}

Additionally, the functions in $H^{s,p}(\R^d)$ are RKBS, when the parameter $s$ is sufficiently high.

\begin{lemma}\label{lemma:bessel_rkbs}
$H^{s,p}(\R^d)$ is an RKBS if and only if $s>d/p$. 
\end{lemma}
\begin{proof}
The spaces $H^{s,p}(\R^d)$ are translation invariant. Hence, the point evaluation functionals $\delta_\bullet$ are bounded if and only if $\delta_0$ is bounded. $\delta_0$ is bounded if and only if
\begin{align*}
    \norm{\delta_0}_{(H^{s,p})^\ast} 
    &= \sup_{\norm{f}_{H^{s,p}}=1}\abs{\braket{\delta_0}{f}} 
    \\&= \sup_{\norm{f}_{H^{s,p}}=1}\abs{f(0)} 
    \\&= \sup_{\norm{\J_{-s}f}_{L^{p}}=1}\abs{\J_{s}\J_{-s}f(0)} & \eqref{eq:bessel_norm},\, \J_\bullet\text{ semigroup}
    \\&= \sup_{\norm{g}_{L^{p}}=1}\abs{\J_{s}g(0)} & g=\J_{-s}f
    \\&= \sup_{\norm{g}_{L^{p}}=1}\abs{G_{s}\ast g(0)} & \eqref{eq:J_G-identity}
    \\&= \sup_{\norm{g}_{L^{p}}=1}\abs{\braket{G_s}{g}_{(L^p)^\ast,L^p}} 
    \\&= \norm{G_s}_{(L^p)^\ast} 
    \\&= \norm{G_s}_{L^{p'}} 
\end{align*}
which holds if and only if $s>d/p$.

Assume now that for an $H^{s,p}(\R^d)$ that $f(x)=0$ for all $x\in \R^d$. By Lemma~\ref{lemma:schwartz_embbedings}, $f\in \S'$ and thus
\begin{equation}
    \braket{f}{\psi}_{\S',\S} = \int f(x)\psi(x) dx = 0
\end{equation}
for all $\psi\in \S$. Since $\S$ separates $\S'$, $f=0$ in $\S'(\R^d)$. By Lemma~\ref{lemma:schwartz_embbedings}, this must mean $f=0$ in $H^{s,p}(\R^d)$.
\end{proof}

Last, different Bessel potential spaces relate to each other.

\begin{lemma}[Sobolev-Bessel embedding theorem, Theorem 14.7.1 of \cite{hytonenAnalysisBanachSpaces2023}]\label{lemma:sobolev_embedding}
Let $u,v\in\R$ and $1<p,q<\infty$. $H^{u,p}(\R^d)\hookrightarrow H^{v,q}(\R^d)$ if and only if $p\leq q$ and $u-d/p \geq v-d/q$. 
\end{lemma}

\section{Proof of main results}\label{sec:proof}
The definition of the RKBS pair is symmetric. Hence, without loss of generality we can assume that $p\leq q$. 

For $H^{u,p}$ and $H^{v,q}$ to be an RKBS pair, they need both to be RKBS. Lemma~\ref{lemma:bessel_rkbs} shows this exactly the case if and only if the lower bounds in \eqref{eq:main_u_result} and \eqref{eq:main_v_result} hold.

For $H^{u,p}$ and $H^{v,q}$ to be an RKBS pair, they also need to be a dual pair. This means that there must exist a separating bilinear map. Since $\S(\R^d)\subset H^{u,p}(\R^d)$ and $\S(\R^d)\subset H^{v,q}(\R^d)$, any pairing $B$ on them can be restricted to $\S(\R^d)\times \S(\R^d)$. If $f\in H^{u,p}(\R^d)$ and the reproducing property holds, then $f\in \S'(\R^d)$ and 
\begin{align*}
    \braket{f}{\psi}_{\S(\R^d)} 
    &= \int f(x)\psi(x)dx                           & f\in C_b(\R^d)
    \\&= \int B(f,K_s(x,\cdot))\psi(x)dx            & \eqref{eq:reproducing_kernel_B}
    \\&= \int B(f,G_{2s}(\cdot-x))\psi(x)dx         & \eqref{eq:kernel}
    \\&= \int B(f,\tau_xG_{2s})\psi(x)dx            & \tau_x f = f(\cdot-x)
    \\&= B\bigg(f,\int \tau_xG_{2s}\;\psi(x)dx\bigg)  & \text{cont. of }B
    \\&= B\bigg(f,G_{2s}\ast \psi\bigg)             
    \\&= B\bigg(f,\J_{2s}\psi\bigg)                 & \eqref{eq:J_G-identity}   
\end{align*}
holds for all $\psi\in \S(\R^d)$. Since $\J_{2s}$ maps $\S(\R^d)$ to $\S(\R^d)$ isometrically, this implies
\begin{equation}\label{eq:bilinear_form1}
    B(f,\psi) = \braket{f}{\J_{-2s}\psi}_{\S(\R^d)}
\end{equation}
for all $f\in H^{u,p}(\R^d)$ and $\psi\in \S(\R^d)$, and thus, in particular, for $(f,\psi)\in \S(\R^d)\times\S(\R^d)$. An analogous argument using the reproducing property of $H^{v,q}(\R^d)$ implies that $B$ must satisfy
\begin{equation}\label{eq:bilinear_form2}
    B(\psi,g) = \braket{g}{\J_{-2s}\psi}_{\S(\R^d)}
\end{equation}
for all $g\in H^{v,q}(\R^d)$ and $\psi\in \S(\R^d)$, and thus, in particular, for $(\psi,g)\in \S(\R^d)\times\S(\R^d)$. Since $J_{-2s}$ is self-adjoint, \eqref{eq:bilinear_form1} and \eqref{eq:bilinear_form2} are consistent with each other. 

The existence of any continuous pairing $\braket{\cdot}{\cdot}$ on $H^{u,p}(\R^d)\times H^{v,q}(\R^d)$ with the reproducing property requires $\J_{-2s}\colon \S(\R^d)\subset H^{u,p}(\R^d)\to (H^{v,q}(\R^d))^\ast$ to be bounded, because
\begin{align*}
    \norm{\J_{-2s}\psi}_{(H^{v,q})^\ast} 
    &= \sup_{\norm{g}_{H^{v,q}}=1}\abs{\braket{\J_{-2s}\psi}{g}_{(H^{v,q})^\ast,H^{v,q}}} & \text{def. dual norm}
    \\&= \sup_{\norm{g}_{H^{v,q}}=1}\abs{\braket{g}{\J_{-2s}\psi}_{\S',\S}} & g\in \S', \J_{-2s}\psi \in \S
    \\&= \sup_{\norm{g}_{H^{v,q}}=1}\abs{\braket{\psi}{g}} & \eqref{eq:bilinear_form2}
    \\&\lesssim \norm{\psi}_{H^{u,p}} & \text{cont. of }\braket{\cdot}{\cdot}
\end{align*}
holds for all $\psi\in \S(\R^d)\subset H^{u,p}(\R^d)$. 

For $p=\infty$, this cannot be satisfied. If $p=\infty$, then by assumption $q=\infty$. Choose any $\phi\in \S(\R^d)$ with $\hat{\phi}(0)\neq 0$ and set 
\begin{equation}\label{eq:counter_example_p_infty}
    h_R(x) = \phi(x/R)
\end{equation}
for all $R>0$. This function is in $\S(\R^d)$ satisfies 
\begin{align*}
    \norm{\J_{-2s}h_R}_{(H^{v,\infty})^\ast} 
    &= \sup_{\norm{g}_{H^{v,\infty}}=1}\abs{\braket{\J_{-2s}h_R}{g}} 
    \\&= \sup_{\norm{\J_{-v}g}_{L^{\infty}}=1}\abs{\braket{\J_{-2s}h_R}{\J_{v}\J_{-v}g}} & \eqref{eq:bessel_norm},\,\J_\bullet\text{ semigroup}
    \\&= \sup_{\norm{\psi}_{L^{\infty}}=1}\abs{\braket{\J_{-2s}h_R}{\J_{v}\psi}}  & \psi=\J_{-v}g
    \\&= \sup_{\norm{\psi}_{L^{\infty}}=1}\abs{\braket{\J_{v-2s}h_R}{\psi}} & \J_\bullet\text{ self-adjoint}
    \\&= \norm{\J_{v-2s}h_R}_{L^1}
    \\&\geq \widehat{\J_{u+v-2s}h_n}(0)   & \norm{\hat{\cdot}}_{L^\infty} \leq \norm{\cdot}_{L^1}
    \\&= (1+\abs{0}^2)^{-(u+v-2s)/2}\widehat{h_R}(0) & \eqref{eq:J}
    \\&= \widehat{h_R}(0)
    \\&= R^{d}\widehat{\phi}(0) & \text{Fourier scaling property}
    \\&\longrightarrow \infty
\end{align*}
as $R\to\infty$. Hence, $\J_{-2s}\colon \S(\R^d)\subset H^{u,p}(\R^d)\to (H^{v,q}(\R^d))^\ast$ is unbounded for $p=\infty$. 

We assume $p<\infty$ from here on. For $p<\infty$, the boundedness of $\J_{-2s}$ is also sufficient for a continuous pairing $\braket{\cdot}{\cdot}$ on $H^{u,p}(\R^d)\times H^{v,q}(\R^d)$ with the reproducing property to exist. If $\J_{-2s}$ is bounded, then define a bilinear map $\braket{\cdot}{\cdot}$ on $\S(\R^d)\times H^{v,q}(\R^d)$ by \eqref{eq:bilinear_form2}. This bilinear map satisfies
\begin{align*}
    \abs{\braket{\psi}{g}}
    &= \abs{\braket{g}{\J_{-2s}\psi}_{\S',\S}}
    \\&= \abs{\braket{\J_{-2s}g}{\psi}_{\S',\S}} & \J_{-2s}\text{ self-adjoint}
    \\&= \abs{\braket{\J_{-2s}\psi}{g}_{(H^{v,q})^\ast,H^{v,q}}} & \J_{-2s}g\in \S',\psi\in \S
    \\&\leq \norm{J_{-2s}\psi}_{(H^{v,q})^\ast}\norm{g}_{H^{v,q}} &\text{cont. of }\braket{\cdot}{\cdot}_{(H^{v,q})^\ast,H^{v,q}}
    \\&\leq \norm{J_{-2s}}\norm{\psi}_{H^{u,p}}\norm{g}_{H^{v,q}} &\J_{-2s}\text{ bounded}
\end{align*}
for all $g\in H^{v,q}(\R^d)$ and $\psi\in \S(\R^d)\subset H^{u,p}(\R^d)$. Since $\S(\R^d)$ is dense in $H^{u,p}(\R^d)$ by Lemma~\ref{lemma:schwartz_embbedings}, $\S(\R^d)\times H^{v,q}(\R^d)$ is dense in  $H^{u,p}(\R^d) \times H^{v,q}(\R^d)$ and there exists a unique extension of $\braket{\cdot}{\cdot}$ which is continuous on $H^{u,p}(\R^d) \times H^{v,q}(\R^d)$.

Since $p<\infty$, $\J_{-2s}$ extends uniquely to $\J_{-2s}\colon H^{u,p}(\R^d)\to (H^{v,q}(\R^d))^\ast$ by density of $\S(\R^d)$ in $H^{u,p}(\R^d)$. Since $\J_\bullet$ is an isometry, we get that this extension is bounded if and only if $H^{u-2s,p}(\R^d)\hookrightarrow (H^{v,q}(\R^d))^\ast$. We will show that this holds if and only if \eqref{eq:main_p_q_result} and \eqref{eq:main_u+v_result} hold. We consider the cases $q<\infty$ and $q=\infty$ separately. However, first we observe that $\J_{-2s}$ bounded if and only if $\norm{\J_{u+v-2s}f}_{(L^q)^\ast}\lesssim \norm{f}_{L^p}$ for all $f\in L^{p}(\R^d)$, because
\begin{align*}
    \norm{f}_{L^p} 
    &=\norm{\J_{u}f}_{H^{u,p}}& \eqref{eq:bessel_norm}
    \\&\gtrsim  \norm{\J_{u-2s}f}_{(H^{v,q})^\ast} & \J_{-s}\text{ bounded}
    \\&= \sup_{\norm{g}_{H^{v,q}}=1}\abs{\braket{\J_{u-2s}f}{g}_{(H^{v,q})^\ast,H^{v,q}}}
    \\&= \sup_{\norm{\J_{-v}g}_{L^{q}}=1}\abs{\braket{\J_{u-2s}f}{\J_{v}\J_{-v}g}_{(H^{v,q})^\ast,H^{v,q}}} & \eqref{eq:bessel_norm},\,\J_\bullet\text{ semigroup}
    \\&= \sup_{\norm{h}_{L^{q}}=1}\abs{\braket{\J_{u-2s}f}{\J_{v}h}_{(H^{v,q})^\ast,H^{v,q}}} & h= \J_{-v}g
    \\&= \sup_{\norm{h}_{L^{q}}=1}\abs{\braket{\J_{u+v-2s}f}{h}_{(L^{q})^\ast,L^{q}}} & \J_\bullet\text{ self-adjoint}
    \\&=\norm{\J_{u+v-2s}f}_{(L^{q})^\ast}
\end{align*}
Any counter examples in the following deal with $f\in L^{p}(\R^d)$ such that the inequality $\norm{\J_{u+v-2s}f}_{(L^q)^\ast}\lesssim \norm{f}_{L^p}$ does not hold.

For $q<\infty$, Lemma~\ref{lemma:dual} implies that $(H^{v,q}(\R^d))^\ast = H^{-v,q'}(\R^d)$. If $p>1$, then both $p,q\in (1,\infty$) and then Lemma~\ref{lemma:sobolev_embedding} shows that the embedding holds if and only if
\begin{subequations}
\begin{align}
    p &\leq q' \\
    u -2s - d/p &\geq -v-d/q'
\end{align}
\end{subequations}
This can be rewritten into \eqref{eq:main_p_q_result} and \eqref{eq:main_u+v_result}. If $p=1$, then $\min(p,q)=1$ and $\max(p,q)<\infty$. If \eqref{eq:main_u+v_result} holds with strict inequality, then $G_{u+v-2s}\in L^{q'}(\R^d)$ and
\begin{align*}
    \norm{\J_{u+v-2s}f}_{(L^{q})^\ast} 
    &= \norm{G_{u+v-2s}\ast f}_{L^{q'}} & \eqref{eq:J_G-identity}
    \\&\leq \norm{G_{u+v-2s}}_{L^{q'}}\norm{f}_{L^1} & \text{Young's conv. ineq.}
\end{align*}
shows that $J_{-2s}$ is bounded. If \eqref{eq:main_u+v_result} doesn't hold with strict inequality, then we consider two cases: $u+v-2s \leq 0$ and $u+v-2s\in (0,d/q']$. The latter is only relevant for $q>1$. If $u+v-2s \leq 0$, then there exists an $\alpha\in (0,2s-u-v)$. Take any such $\alpha$ and observe that $G_\alpha\in L^1(\R^d)$ by Lemma~\ref{lemma:G_integrability}. However, $\J_{u+v-2s}G_\alpha$ is only a distribution and, in particular, not in $L^{q'}(\R^d)$. Hence, $\J_{-2s}$ cannot be bounded for $p=1$ and $u+v<2s$. If $u+v-2s\in (0,d/q']$, then the operator is unbounded by a standard mollifier approach. Choose any nonzero and nonnegative $\phi\in \S(\R^d)$ with $\norm{\phi}_{L^1}=1$ and set 
\begin{equation}
    h_\epsilon(x) = \epsilon^{-d}\phi(x/\epsilon)
\end{equation}
for all $\epsilon>0$. This makes $h_\epsilon$ an approximate identity. Hence,
\begin{equation}
    G_{u+v-2s}\ast h_{\epsilon} \xrightarrow{\epsilon\to 0}G_{u+v-2s}
\end{equation}
pointwise. Suppose for contradiction that the operator $\J_{-2s}$ is bounded. This implies
\begin{align*}
    \norm{G_{u+v-2s}}_{L^{q'}} 
    &\leq \liminf_{\epsilon\to 0}\norm{G_{u+v-2s}h_{\epsilon}}_{L^{q'}} & \text{Fatou's lemma}
    \\&= \norm{\J_{u+v-2s} h_{\epsilon}}_{L^{q'}}   & \eqref{eq:J_G-identity}
    \\&= \norm{\J_{u-2s}h_{\epsilon}}_{H^{-v,q'}}   & \eqref{eq:bessel_norm}
    \\&\lesssim \norm{\J_{u}h_{\epsilon}}_{H^{u,1}} & \J_{-2s}\text{ bounded}
    \\&= \norm{h_{\epsilon}}_{L^{1}}                & \eqref{eq:bessel_norm}
    \\&= 1                                          & h_\epsilon\text{ approx. id.}
\end{align*}
and thus $G_{u+v-2s}\in L^{q'}$. However, $G_{u+v-2s}\not\in L^{q'}$ for the given $u+v-2s$ according to Lemma~\ref{lemma:G_integrability}. By contradiction, $\J_{-2s}$ must be unbounded for $p=1$ and $u+v-2s\in (0,d/q']$.

For $q=\infty$, we consider the cases $p>1$ and $p=1$. If $p>1$, then consider an $\psi\in\S(R^d)$ such that $\hat{\psi}(0)\neq 0$ and define the sequence $h_\bullet$ by $h_n(x) = n^{-d/p}\psi(x/n)$ for all $n\in \N$. This sequence is constant in $L^p(\R^d)$, i.e. $\norm{h_n}_{L^p}=\norm{\psi}_{L^p}$, because
\begin{equation}
    \norm{h_n}^p_{L^p} 
    = \int_{\R^d}\abs{h_n(x)}^pdx
    = \int_{\R^d}n^{-d}\abs{\psi(x/n)}^pdx
    = \int_{\R^d}\abs{\psi(y)}^pdy
    = \norm{\psi}_{L^p}^p < \infty
\end{equation}
However, 
\begin{align*}
    \norm{\J_{u+v-2s}h_n}_{L^1} 
    &\geq \widehat{\J_{u+v-2s}h_n}(0)   & \norm{\hat{\cdot}}_{L^\infty} \leq \norm{\cdot}_{L^1}
    \\&= (1+\abs{0}^2)^{-(u+v-2s)/2}\widehat{h_n}(0) & \eqref{eq:J}
    \\&= \widehat{h_n}(0)
    \\&= n^{d(1-1/p)}\widehat{\psi}(0)  & \text{Fourier scaling property}
    \\&\longrightarrow \infty
\end{align*}
as $n\to\infty$. Hence, $\J_{-2s}$ cannot be bounded for $p>1$. If $p=1$ and $u+v<2s$, then there exists an $\alpha\in (0,2s-u-v)$. Take any such $\alpha$ and observe that $G_\alpha\in L^1(\R^d)$ by Lemma~\ref{lemma:G_integrability}. However, $\J_{u+v-2s}G_\alpha$ is only a distribution and, in particular, not in $L^1(\R^d)$. Hence, $\J_{-2s}$ cannot be bounded for $p=1$ and $u+v<2s$. If $p=1$ and $u+v> 2s$, then Young's inequality implies that
\begin{equation}
    \norm{\J_{u+v-2s}h}_{L^1} = \norm{G_{u+v-2s}\ast h}_{L^1} \leq \norm{G_{u+v-2s}}_{L^1}\norm{h}_{L^1}
\end{equation}
with $\norm{G_{u+v-2s}}_{L^1}$ finite because $G_{u+v-2s}\in L^1(\R^d)$ by Lemma~\ref{lemma:G_integrability}. If $p=1$ and $u+v=2s$, then $\J_{u+v-2s}=Id$ and the bound holds trivially. Hence, $\J_{-2s}$ is bounded for $p=1$ and $u+v\geq 2s$. This finishes the if and only if, and the conditions are the same as \eqref{eq:main_p_q_result} and \eqref{eq:main_u+v_result}.

What remains is showing that the pairing induced by $\J_{-2s}$ satisfies the reproducing properties and that $K_s$ is the kernel for the RKBS pair $H^{u,p}(\R^d)$ and $H^{v,q}(\R^d)$. Since the Bessel potential spaces are translation-invariant and the kernel symmetric, it is sufficient for the latter to show that $\J_{-u}G_{2s}=G_{2s-u}\in L^p$ and $\J_{-v}G_{2s}=G_{2s-v}\in L^q$. Lemma~\ref{lemma:sobolev_embedding} implies that these are true if and only if 
\begin{subequations}
\begin{align}
    2s-u &> d/p' \\
    2s-v &> d/q'
\end{align}
\end{subequations}
which corresponds to the upper bounds for $u$ and $v$ in \eqref{eq:main_u_result} and \eqref{eq:main_v_result}.
The continuous pairing we derived satisfies the reproducing property for $H^{u,p}(\R^d)$, because $K_s(x,\cdot)\in H^{v,q}(\R^d)$ and thus
\begin{align*}
    \braket{\psi}{K_s(x,\cdot)} 
    &= \braket{K_s(x,\cdot)}{\J_{-2s}\psi}_{\S',\S} & \eqref{eq:bilinear_form2}
    \\&= \braket{\J_{-2s}K_s(x,\cdot)}{\psi}_{\S',\S}   &   \J_\bullet\text{ self-adjoint}
    \\&= \braket{\J_{-2s}G_{2s}(\cdot-x)}{\psi}_{\S',\S} & \eqref{eq:kernel}
    \\&= \braket{\J_{-2s}\tau_xG_{2s}}{\psi}_{\S',\S}   & \tau_x f = f(\cdot-x)
    \\&= \braket{\tau_x\J_{-2s}G_{2s}}{\psi}_{\S',\S}   & \tau_x,\J_\bullet\text{ commute} 
    \\&= \braket{\tau_x\delta_0}{\psi}_{\S',\S} & \J_\bullet\text{ semigroup}
    \\&= \braket{\delta_x}{\psi}_{\S',\S}
    \\&= \psi(x)
\end{align*}
holds for all $\psi\in\S(\R^d)$ and the density of $\S(\R^d)$ in $H^{u,p}(\R^d)$ implies this extends to all $f\in H^{u,p}(\R^d)$. The continuous pairing we derived satisfies the reproducing property for $H^{v,q}(\R^d)$, because $K_s(\cdot,y)\in H^{u,p}(\R^d)$ and thus
\begin{align*}
    \braket{K_s(\cdot,x)}{g} 
    &= \lim_{n\to\infty}\braket{\psi_n}{g}  & \psi_n\to K_s(\cdot,y)
    \\&= \lim_{n\to\infty}\braket{g}{\J_{-2s}\psi_n}_{\S',\S}   & \eqref{eq:bilinear_form2}
    \\&= \lim_{n\to\infty}\braket{\J_{-2s}g}{\psi_n}_{\S',\S}   & \J_\bullet\text{ self-adjoint}
    \\&= \lim_{n\to\infty}\braket{\J_{-2s}g}{\psi_n}_{H^{-u,p'},H^{u,p}}   & \J_{-2s}g\in H^{v+2s,q}\hookrightarrow H^{-u,p'}=(H^{u,p})^\ast
    \\&= \braket{\J_{-2s}g}{K_s(\cdot,y)}_{H^{-u,p'},H^{u,p}}   & \psi_n\to K_s(x,\cdot)
    \\&= \braket{\J_{-2s}g}{G_{2s}(\cdot-y)}_{H^{-u,p'},H^{u,p}}   & \eqref{eq:kernel},\; G_{\bullet}\text{ symmetric}
    \\&= \braket{\J_{-2s}g}{\tau_y G_{2s}}_{H^{-u,p'},H^{u,p}}   & \tau_y f = f(\cdot-y)
    \\&= \braket{\J_{u-2s}g}{\J_{-u}\tau_y G_{2s}}_{L^{p'},L^{p}}   & \eqref{eq:bessel_norm},\, \J_\bullet\text{ semigroup}
    \\&= \braket{\J_{u-2s}g}{\tau_y \J_{-u}G_{2s}}_{L^{p'},L^{p}}   & \tau_x,\J_\bullet\text{ commute}
    \\&= \braket{\J_{u-2s}g}{\tau_y G_{2s-u}}_{L^{p'},L^{p}}   & \J_\bullet\text{ semigroup},\,\eqref{eq:J_G-identity}
    \\&= G_{2s-u}\ast \J_{u-2s}g(x)   & 
    \\&= g(x)   & \eqref{eq:J_G-identity},\, \J_\bullet\text{ semigroup}
\end{align*}
holds for all $g\in H^{v,q}(\R^d)$ and sequences $\psi_n\to K_s(\cdot,y)$ in $H^{u,p}(\R^d)$.

\section{Discussion}
In the work, we answered which distinct dual pairs of RKBSs $H^{u,p}(\R^d),H^{v,q}$ have $K_s(x,y)=G_{2s}(x-y)$ as kernel. Several things are noteworthy here.

\paragraph{Lower bound for $s$}
For the Bessel potential RKHS with kernel $K_s$ we need to have $s>d/2$. The work by Fasshauer et al. has the same bound. Although our main results in Theorem~\ref{th:main_results} don't state an explicit requirement on $s$, the bound $s>d/2$ still is needed. When $s\leq d/2$, \eqref{eq:main_u_result} and \eqref{eq:main_v_result} cannot be satisfied for any $u,v$. That is it needed should not come as a great surprise, since $s>d/2$ corresponds to the case that $G_{2s}\in L^1(\R^d)$, which is a fact most lemmata rely on. 

\paragraph{Non-uniqueness of the kernel for a given pair}
Moore-Aronzaijn established a connection between kernels and reproducing kernel Hilbert spaces. Theorem~\ref{th:main_results} shows that for RKBS pairs not even two RKBSs establish a unique kernel. For fixed $(u,p,v,q)$, we can rewrite the conditions on $s$. This gives the interval
\begin{equation}
    \frac{1}{2}\max(u+\frac{d}{p'},v+\frac{d}{q'}) < s \leq \frac{u+v-d(1/q+1/p-1)}{2}
\end{equation}
with the second inequality strict for $\min(p,q)=1$ and $\max(p,q)<\infty$. Whenever this interval contains more than one point, the pair $H^{u,p}(\R^d),H^{v,q}(\R^d)$ has multiple kernels. For example, the spaces $H^{3,2}(\R)$ and $H^{3,2}(\R)$ form an RKBS pair with kernels $K_2$ and $K_3$, and the spaces $H^{2,1}(\R)$ and $H^{2,1}(\R)$ form an RKBS pair with kernels $K_{5/4}$ and $K_{6/5}$.

\paragraph{Non-Hölder pairs}
In \cite{zhang_reproducing_2009}, the reflexive RKBS are formed use the full dual or the Hölder conjugate space for the diamond space. Theorem~\ref{th:main_results} shows that this is not needed. For example, the spaces $H^{3,1}(\R)$ and $H^{2,2}(\R)$ form an RKBS pair with kernel $K_2$. 

\paragraph{Non-Hilbertian self-pairs}
Reproducing kernel Hilbert spaces can be identified as a special kind of RKBS pair, because $\H,\H$ is an RKBS pair when $\H$ is an RKHS. Theorem~\ref{th:main_results} shows that there are RKBSs which are not RKHS that form an RKBS pair with themselves. Examples include the aforementioned $H^{2,1}(\R)$ and $H^{3,2}(\R)$. For a pair of Bessel potential spaces to be an RKBS \textit{self-pair}, i.e. an RKBS pair with $\B=\B^\di$, the simplified conditions
\begin{subequations}
\begin{align}
    p&\leq 2 \\
    d/p &< u < 2s-d/p' \\
    2u &\geq 2s + d(2/p-1) \label{eq:self_dual_r_s_bound}
\end{align}
\end{subequations}
need to hold, and the interval of admissible $s$ given fixed $(u,p)$ simplifies to
\begin{equation}\label{eq:self_dual_interval}
    \frac{u}{2}+\frac{d}{2p'} < s \leq u-d(\frac{1}{p}-\frac{1}{2})
\end{equation}
where the inequalities in both \eqref{eq:self_dual_r_s_bound} and \eqref{eq:self_dual_interval} are strict for $p=1$.
The admissible range of $s$ for $H^{u,2}(\R^d)$ is given by
\begin{equation}
    u/2 + d/4 < s \leq u
\end{equation}
and the kernel of the RKHS $H^{u,2}(\R^d)$ corresponds to the upper limit of this, i.e. it has $s=u$. For $p<2$, the upper bound of \eqref{eq:self_dual_interval} is strictly less than $u$. Hence, the RKHS $H^{u,2}(\R^d)$ is the only self-pair that has allows for a kernel with the same regularity as the space. 

\paragraph{Norming dual pairs}
A dual pair $\B,\B^\di$ is called \textit{norming} when the norms satisfy
\begin{subequations}
\begin{align}
    \norm{f}_\B &= \sup_{\norm{g}_{\B^\di}=1}\abs{\braket{g}{f}} \\
    \norm{g}_{\B^\di}  &= \sup_{\norm{f}_{\B}=1}\abs{\braket{g}{f}}
\end{align}    
\end{subequations}
Assuming $p,q\in(1,\infty)$, $H^{u,p}(\R^d)$ and $H^{v,q}(\R^d)$ is a norming pair if and only if
\begin{subequations}
\begin{align}
    p &= q'\\
    u+v &= 2s
\end{align}
\end{subequations}
because
\begin{align*}
    \norm{f}_{H^{u,p}} 
    &= \sup_{\norm{g}_{H^{v,q}}=1}\abs{\braket{g}{f}}           & \text{norming}
    \\&= \sup_{\norm{g}_{H^{v,q}}=1}\abs{\braket{\J_{-2s}g}{f}_{H^{-u,p'},H^{u,p}}}     & \text{extension of }\eqref{eq:bilinear_form2}
    \\&= \sup_{\norm{h}_{H^{v-2s,q}}=1}\abs{\braket{h}{f}_{H^{-u,p'},H^{u,p}}}  &    \eqref{eq:bessel_norm},\,h=\J_{-2s}g
    \\&= \sup_{\norm{h}_{H^{v-2s,q}}=1}\abs{\braket{\J_{u}h}{\J_{-u}f}_{L^{p'},L^{p}}} 
    \\&= \sup_{\norm{h}_{H^{v-2s,q}}=1}\abs{\int h(x)f(x)dx} & \J_\bullet\text{ self-adjoint},\,\J_\bullet\text{ semigroup}
    \\&= \sup_{\norm{h}_{H^{v-2s,q}}=1}\abs{\braket{\J_{2s-v}h}{\J_{v-2s}f}_{L^{q},L^{q'}}}  & \J_\bullet\text{ semigroup},\, \J_\bullet\text{ self-adjoint}
    \\&= \sup_{\norm{h}_{H^{v-2s,q}}=1}\abs{\braket{h}{f}_{H^{v-2s,q},H^{2s-v,q'}}} 
    \\&= \norm{f}_{H^{2s-v,q'}}     & \text{Lemma}~\ref{lemma:dual}
\end{align*}
holds for all $f\in H^{u,p}$ and likewise
\begin{equation}
    \norm{g}_{H^{v,q}} = \norm{g}_{H^{2s-r,p'}}
\end{equation}
for all $g\in H^{v,q}$. These values are exactly the cases when \eqref{eq:main_p_q_result} and \eqref{eq:main_u+v_result} achieve equality. The only norming RKBS self-pair is the pair corresponding to the RKHS. In the general case, choosing two of the triplet $(H^{u,p}(\R^d),H^{v,q}(\R^d),K_s)$ does not imply the third. 
If we only consider norming dual pairs for the RKBS pairs, then we see that a \textit{2-out-of-3} property holds: knowing two of the three of the triplet $(H^{u,p}(\R^d),H^{v,q}(\R^d),K_s)$ fixes the third.

\paragraph{Vanishing at infinity for $\max(p,q)=\infty$}
In Section~\ref{sec:proof}, we distinguished between $q=\infty$ and $q<\infty$. The proof for $q<\infty$ relied on the Schwartz space $\S(\R^d)$ being dense in both $H^{u,p}(\R^d)$ and $H^{v,q}(\R^d)$. For $q=\infty$, the Schwartz space $\S(\R^d)$ is not dense in $H^{v,q}(\R^d)$. The sup-closure of $\S(\R^d)$ is $C_0(\R^d)$ and not $L^\infty(\R^d)$ or $C_b(\R^d)$, since $\S(\R^d)$ consists of smooth functions vanishing at infinity and includes $C_c^\infty(\R^d)$, the smooth compactly-supported functions. Hence, an alternative space we could consider is 
\begin{equation}
    H^{s}_0(\R^d) = \J_{-s}C_0(\R^d) = \set{ f\in \S'(\R^d) \given \exists f_0\in C_0(\R^d)\colon \J_sf_0 = f} 
\end{equation}
Using an analogous proof strategy, we can prove that $H^{u,1}(\R^d)$ and $H^{v,\infty}_0(\R^d)$ under the conditions of Theorem~\ref{th:main_results}.  

\paragraph{Strict inequality for $\max(p,q)<\infty$ and $\min(p,q)=1$}
In Theorem~\ref{th:main_results}, \eqref{eq:main_u+v_result} is strict for $\max(p,q)<\infty$ and $\min(p,q)=1$. This has to do with the relevant part of the proof reducing to asking whether $H^{2s-v,q'}(\R^d)$ embeds into $H^{u,1}(\R^d)$, assuming $p\leq q$. In Lemma~\ref{lemma:sobolev_embedding}, the equality constraint case is proven using an interpolation argument. Here, that doesn't work due to the approximate identities on $L^1(\R^d)$. This strict inequality shows that Lemma~\ref{lemma:sobolev_embedding} does not extend with inequality to $1\leq p,q<\infty$, unlike the case for Sobolev spaces $W^{k,p}(\R^d)$.

\paragraph{Comparison to $\ell^p,\ell^q$ RKBS pairs}
In the previous paragraphs, we discussed several ways in which the RKBS pairs for the Bessel potential spaces differ from the Bessel potential RKHS. The ways these differ is not limited to just the Bessel potential spaces. Consider as example the $\ell^p$ spaces interpreted as functions from $\N$ to $\R$ and the kernel
\begin{equation}\label{eq:sequence_kernel}
    K_\N\colon \N\times\N\to\R,\; (n,m)\mapsto \begin{cases}
        1 & n=m \\ 0 & n\neq m
    \end{cases}
\end{equation}
The spaces $\ell^p$ and $\ell^q$ form an RKBS pair with kernel $K_\N$ if and only if
\begin{equation}
    \frac{1}{p}+\frac{1}{q}\geq 1
\end{equation}
We see that $\ell^1$ and $\ell^3$ form a non-Hölder RKBS pair, $\ell^p$ with $p\leq 2$ form RKBS self-pairs, $\ell^p$ and $\ell^q$ are a norming RKBS pair if and only if $q$ is the Hölder conjugate of $p$, and the space $c_0$ is an alternative space for $\ell^\infty$. 

\section{Conclusion}\label{sec:conclusion}
This work deals with the question of which Bessel potential spaces $H^{u,p}(\R^d)$ have the kernel
\begin{equation}
    K_s(x,y) = G_{2s}(x-y)
\end{equation}
as reproducing kernel. Theorem~\ref{th:main_results} provides a full characterization of the pairs of Bessel potential spaces $H^{u,p}(\R^d)$ and $H^{v,q}(\R^d)$ have $K_s$ as kernel. 

\textbf{Future work}
Theorem~\ref{th:main_results} shows that the relation between kernel and spaces is non-unique: given pairs of Bessel potential spaces can be equipped with multiple kernels by choosing a different pairing, and every kernel can serve as kernel for multiple pairs of Bessel potential spaces. This leaves open the question of which to choose in practice. Understanding the impact of the choice on the representer theorem and, in particular, whether the solutions to optimization problems of the form
\begin{equation}
    \min_{f\in H^{u,p}(\R^d)}\sum_{n=1}^\infty \abs{f(x_n)-y_n}^2 + \lambda \norm{f}_{H^{u,p}}^p
\end{equation}
have meaningfully different solution, would provide clarity on these choices. 

Next, Theorem~\ref{th:main_results} only considers the Bessel potential spaces. These spaces are a special case of the larger family of Triebel-Lizorkin $F^{s}_{p,q}(\R^d)$: The Bessel potential spaces are isomorph to the special case $q=2$. It is natural to ask which Triebel-Lizorkin spaces admit $K_s$ as kernel. 

Last, Theorem~\ref{th:main_results} shows that $K_s$ is the kernel of a continuum of Banach space pairs. This raises the question of whether every positive definite kernel admits a similarly rich family. The sequence kernel $K_\N$ defined in \eqref{eq:sequence_kernel} suggests the latter: there the admissible $\ell^p,\ell^q$ pairs are characterized by the inequality $1/p + 1/q \geq 1$. Determining whether such a clean characterization persists for kernels without an explicit convolution structure would shed light on the extent to which Theorem~\ref{th:main_results} reflects a universal property of reproducing kernels.

\printbibliography

\end{document}